\documentclass[11pt,a4paper]{article}

\usepackage[utf8]{inputenc}
\usepackage[T1]{fontenc}
\usepackage{geometry}
\usepackage{amsmath,amssymb,amsthm,mathrsfs,mathtools,bm}
\usepackage{microtype}
\usepackage{enumitem}
\usepackage[colorlinks=true,linkcolor=blue,citecolor=blue,urlcolor=blue]{hyperref}
\usepackage{hyphenat}
\numberwithin{equation}{section}

\newtheorem{theorem}{Theorem}[section]
\newtheorem{corollary}[theorem]{Corollary}
\newtheorem{lemma}[theorem]{Lemma}
\newtheorem{proposition}[theorem]{Proposition}
\newtheorem{remark}[theorem]{Remark}

\newcommand{\R}{\mathbb{R}}
\newcommand{\C}{\mathbb{C}}
\newcommand{\Hn}{\mathbb{H}^n}
\newcommand{\Sp}{S^{2n+1}}
\newcommand{\om}{\omega_{2n+1}}
\newcommand{\qexp}{p^*}
\newcommand{\dd}{\,\mathrm{d}}

\newcommand{\Bsph}{\mathbb{B}_p\bigl(HW^{s,p}(\Sp)\bigr)}
\newcommand{\Bhsph}{\widehat{\mathbb{B}}_p\bigl(HW^{s,p}(\Sp)\bigr)}
\newcommand{\Lip}{\mathrm{Lip}}

\begin{document}

\title{Fractional Sobolev-type embedding on CR sphere and Heisenberg group}

\author{Zongxiong Ren, Zhipeng Yang\thanks{Corresponding author: yangzhipeng326@163.com.}}

\date{}

\AtEndDocument{%
  \par
  \bigskip
  \bigskip

  \noindent
  \textbf{Zongxiong Ren}\\[0.2em]
  \textsc{Department of Mathematics, Yunnan Normal University, Kunming, China}\\[0.3em]
  \textit{E-mail address}: \texttt{2448783498@qq.com}\\[1.5em]
  
  \noindent
  \textbf{Zhipeng Yang}\\[0.2em]
  \textsc{Department of Mathematics, Yunnan Normal University, Kunming, China}\\
  \textsc{Yunnan Key Laboratory of Modern Analytical Mathematics and Applications, Kunming, China}\\[0.3em]
  \textit{E-mail address}: \texttt{yangzhipeng326@163.com}
 
}

\date{}
\maketitle

\begin{abstract}
This paper studies critical fractional Sobolev inequalities with lower-order terms on the standard CR sphere $\Sp$. Let $Q=2n+2$, let $s\in(0,1)$, let $1<p<Q$, and let $\qexp=Qp/(Q-sp)$. For the inequality $\|u\|_{L^{\qexp}(\Sp)}\le A[u]_{s,p}+B\|u\|_{L^p(\Sp)}$, we prove that the admissible lower-order coefficients are exactly $[\om^{-s/Q},\infty)$. For the power-type inequality $\|u\|_{L^{\qexp}(\Sp)}^p\le A[u]_{s,p}^p+B\|u\|_{L^p(\Sp)}^p$, we show that the admissible set is $[\om^{-sp/Q},\infty)$ when $1<p\le2$, and $(\om^{-sp/Q},\infty)$ when $2<p<Q$. Via the Cayley transform, we derive the exact weighted counterpart on the Heisenberg group and prove that the corresponding admissible sets coincide with those on the sphere. We also show that nonlinear first-moment constraints do not improve the optimal lower-order coefficient, whereas finite-codimensional linear constraints excluding nonzero constants yield coercive inequalities.
\end{abstract}


\section{Introduction and main results}

Critical Sobolev inequalities and their refined forms are a basic tool in geometric analysis, the calculus of variations, and nonlinear partial differential equations; see, for instance,
\cite{MR1617413,MR473443,MR494315,MR657581,MR2363343,MR3469687,MR3966452}.
On compact manifolds, the critical inequality is naturally accompanied by a lower-order term, and the determination of the optimal coefficient of this remainder is closely related to the program initiated by Hebey \cite{MR1688256}. In subelliptic geometry, the corresponding role is played by the Folland--Stein spaces on Carnot groups and CR manifolds. Fractional versions on the Heisenberg group and on more general Carnot groups have attracted considerable attention in recent years; we refer to
\cite{MR3807591,kassymov2018,MR4091395,MR5001434,MR4721790}
and the references therein.

The present paper is concerned with the critical fractional Sobolev inequality on the standard CR sphere and, more specifically, with the lower-order term that necessarily appears in the compact setting. Our aim is twofold. First, we determine the admissible range of the lower-order coefficient in the critical fractional inequality on the standard CR sphere $\Sp$. Second, using the Cayley transform, we identify the corresponding weighted formulation on the Heisenberg group $\Hn$ and show that it has exactly the same admissible range.

The sphere is the natural compact counterpart of the Heisenberg group. It carries its standard pseudohermitian structure $\theta_0$, while $\Sp\setminus\{S\}$ is CR equivalent to $\Hn$ through the Cayley transform, where $S=(0,\dots,0,-1)$ denotes the south pole. This compact--noncompact correspondence is central to the present problem. On the one hand, the compactness of $\Sp$ forces a qualitative difference from the homogeneous inequalities on $\Hn$: constant functions belong to the energy space and have vanishing fractional seminorm, so no estimate on the full space can control the whole $L^p$ or $L^{p^*}$ norm solely by the seminorm. On the other hand, when the sphere inequality is transported to $\Hn$, one does not recover the standard unweighted homogeneous inequality, but rather a weighted critical inequality in which both the $L^r$ norms and the fractional kernel are modified by the Cayley Jacobian. A precise formulation of this correspondence is one of the main points of the paper.

Let $\Sp\subset\C^{n+1}$ be endowed with its standard pseudohermitian structure $\theta_0$, and let
\[
\dd V_{\theta_0}=\theta_0\wedge(\dd\theta_0)^n,
\qquad
\om:=\int_{\Sp}\dd V_{\theta_0}.
\]
Throughout the paper, $d$ denotes a fixed CR distance associated with $\theta_0$. Since $\Sp$ is compact, different such distances are equivalent and lead to equivalent seminorms. We write
\[
Q=2n+2,
\qquad
1<p<Q,
\qquad
s\in(0,1),
\qquad
\qexp=\frac{Qp}{Q-sp}.
\]
For a measurable function $u:\Sp\to\R$, we define
\begin{equation}\label{eq:intro-seminorm}
[u]_{s,p}^p
:=
\int_{\Sp}\int_{\Sp}
\frac{|u(\xi)-u(\eta)|^p}{d(\xi,\eta)^{Q+sp}}
\,\dd V_{\theta_0}(\xi)\,\dd V_{\theta_0}(\eta),
\end{equation}
and
\[
HW^{s,p}(\Sp)
:=
\bigl\{u\in L^p(\Sp):[u]_{s,p}<\infty\bigr\},
\qquad
\|u\|_{HW^{s,p}(\Sp)}
:=
\|u\|_{L^p(\Sp)}+[u]_{s,p}.
\]

We study the following two critical inequalities:
\begin{equation}\label{eq:AB1-intro}
\|u\|_{L^{\qexp}(\Sp)}
\le A[u]_{s,p}+B\|u\|_{L^p(\Sp)}
\qquad
\text{for all }u\in HW^{s,p}(\Sp),
\end{equation}
and
\begin{equation}\label{eq:AB2-intro}
\|u\|_{L^{\qexp}(\Sp)}^p
\le A[u]_{s,p}^p+B\|u\|_{L^p(\Sp)}^p
\qquad
\text{for all }u\in HW^{s,p}(\Sp).
\end{equation}
In the language of the A--B program from Hebey \cite{MR1688256}, the coefficient $A$ corresponds to the leading, or homogeneous, part of the inequality, whereas $B$ is the coefficient of the lower-order term. In this paper we focus on the determination of the admissible $B$-coefficients. Accordingly, we introduce
\[
\Bsph
:=
\Bigl\{
B\in\R:\exists\,A\in\R \text{ such that \eqref{eq:AB1-intro} holds for all }
u\in HW^{s,p}(\Sp)
\Bigr\},
\]
and
\[
\Bhsph
:=
\Bigl\{
B\in\R:\exists\,A\in\R \text{ such that \eqref{eq:AB2-intro} holds for all }
u\in HW^{s,p}(\Sp)
\Bigr\}.
\]

The first point is that the compact problem is governed by the interaction between the average of $u$ and its oscillation. Since constants satisfy $[u]_{s,p}=0$, the lower-order term cannot be removed on the whole space. In particular, any admissible value of $B$ must at least compensate for constant test functions. Our first theorem shows that this necessary lower bound is, in fact, sharp for the linear form \eqref{eq:AB1-intro}.

\begin{theorem}\label{thm:main-B-intro}
For every $1<p<Q$ and every $s\in(0,1)$,
\[
\Bsph=[\om^{-s/Q},\infty).
\]
In particular, $\Bsph$ is closed and its left endpoint is attained.
\end{theorem}

For the power form \eqref{eq:AB2-intro}, the picture is slightly subtler. The optimal lower bound is again dictated by constant functions, but its admissibility depends on the range of $p$.

\begin{theorem}\label{thm:main-Bhat-intro}
Let $1<p<Q$ and $s\in(0,1)$.
\begin{enumerate}[label=\textup{(\roman*)}]
\item If $1<p\le2$, then
\[
\Bhsph=[\om^{-sp/Q},\infty).
\]
In particular, $\Bhsph$ is closed and its left endpoint is attained.
\item If $2<p<Q$, then
\[
\Bhsph=(\om^{-sp/Q},\infty).
\]
Hence $\om^{-sp/Q}$ is the optimal threshold, but it is not admissible; in particular, $\Bhsph$ is not closed.
\end{enumerate}
\end{theorem}

The second main theme of the paper is the translation of the compact sphere problem to the Heisenberg group. Let $\Psi_c:\Sp\setminus\{S\}\to \Hn$ be the Cayley transform, with inverse $\Psi_c^{-1}:\Hn\to \Sp\setminus\{S\}$. If $v=u\circ\Psi_c^{-1}$, then for every $r\ge1$ one has
\[
\|u\|_{L^r(\Sp)}^r
=
\int_{\Hn}|v(x)|^r\,J_c(x)\,\dd x,
\]
while the seminorm \eqref{eq:intro-seminorm} is transformed into a weighted Gagliardo-type seminorm on $\Hn$ with kernel $K_c$. Thus the sphere inequalities \eqref{eq:AB1-intro} and \eqref{eq:AB2-intro} are equivalent to weighted critical inequalities on $\Hn$. The point is that the compact and noncompact formulations have exactly the same admissible lower-order coefficients.

\begin{theorem}\label{thm:main-cayley-intro}
Let $J_c$ and $K_c$ be the weight and the kernel induced by the Cayley transform. Then the sphere inequalities \eqref{eq:AB1-intro} and \eqref{eq:AB2-intro} are equivalent to weighted inequalities on $\Hn$ involving the weighted norms
\[
\|v\|_{L^r(J_c)}
:=
\left(\int_{\Hn}|v|^rJ_c\,\dd x\right)^{1/r}
\]
and the transported seminorm $[v]_{K_c,p}$. Moreover, the corresponding admissible $B$-sets on $\Hn$ coincide exactly with $\Bsph$ and $\Bhsph$.
\end{theorem}

We also reconsider constrained versions of the critical inequalities. Here one must distinguish carefully between constraints that merely impose symmetry-type conditions and constraints that actually remove constants. For instance, the nonlinear first-moment conditions
\[
\int_{\Sp}\xi_i |u|^{\qexp}\,\dd V_{\theta_0}=0,
\qquad i=1,\dots,2n+2,
\]
still admit every constant function, because $\int_{\Sp}\xi_i\,\dd V_{\theta_0}=0$ by symmetry. Consequently, these constraints do not improve the optimal lower-order coefficient: constant test functions produce exactly the same lower bound as in the unrestricted case. A genuinely different phenomenon occurs for finite-codimensional linear constraints that exclude non-zero constants. In that case one recovers a coercive Poincar\'e-type estimate and therefore a pure seminorm inequality.

\begin{theorem}\label{thm:main-constraints-intro}
Let $\mathcal X\subset HW^{s,p}(\Sp)$ be a finite-codimensional closed linear subspace such that
\[
\mathcal X\cap\{\text{constants}\}=\{0\}.
\]
Then there exists a constant $C>0$ such that
\[
\|u\|_{L^{\qexp}(\Sp)}\le C[u]_{s,p}
\qquad\text{and}\qquad
\|u\|_{L^{\qexp}(\Sp)}^p\le C^p[u]_{s,p}^p
\quad\text{for all }u\in\mathcal X.
\]
Equivalently, on $\mathcal X$ every real number is admissible as a lower-order coefficient in \eqref{eq:AB1-intro} and \eqref{eq:AB2-intro}.
\end{theorem}

As a further consequence of the endpoint analysis, we obtain a family of subcritical inequalities. More precisely, for every $r\in[p,\qexp)$ and every $\varepsilon>0$, there exists $C_{\varepsilon,r}>0$ such that
\[
\|u\|_{L^r(\Sp)}
\le \varepsilon [u]_{s,p}+C_{\varepsilon,r}\|u\|_{L^p(\Sp)}
\qquad
\text{for all }u\in HW^{s,p}(\Sp).
\]
The corresponding weighted inequalities on $\Hn$ follow by transport through the Cayley transform.

The rest of the paper is organized as follows. In Section \ref{sec:background} we review the CR sphere, the Heisenberg group, and the Cayley transform. Section \ref{sec:prelim} establishes the analytic tools needed on $\Sp$, including the continuous embedding, the Poincar\'e inequality, the zero-average Sobolev inequality, and the cutoff estimates. Sections \ref{sec:B-program} and \ref{sec:Bhat-program} determine the admissible sets $\Bsph$ and $\Bhsph$. In Section \ref{sec:cayley-H} we derive the weighted Heisenberg formulation and prove the equivalence of admissible sets under the Cayley transform. Section \ref{sec:further-constraints} is devoted to constrained inequalities, with particular emphasis on the contrast between first-moment classes and coercive linear constraint classes. The final section records the subcritical consequences and discusses several directions for further study.

\section{Geometric background}\label{sec:background}

This section recalls the geometric framework underlying the subsequent analysis. We begin with the general pseudohermitian setting, then describe the two model spaces relevant to the paper, namely the Heisenberg group and the standard CR sphere, and finally record the Cayley transform that relates the compact and noncompact pictures.

Let $N$ be a smooth manifold of real dimension $2n+1$. A CR structure on $N$ is a complex rank-$n$ subbundle $T^{1,0}N\subset \C TN$ such that
\[
T^{1,0}N\cap T^{0,1}N=\{0\},
\qquad
[\Gamma(T^{1,0}N),\Gamma(T^{1,0}N)]\subset \Gamma(T^{1,0}N),
\]
where $T^{0,1}N=\overline{T^{1,0}N}$. The underlying real horizontal bundle is
\[
HN:=\Re\bigl(T^{1,0}N\oplus T^{0,1}N\bigr).
\]
If there exists a globally defined one-form $\theta$ such that $\ker \theta=HN$ and the Levi form
\[
L_\theta(V,\overline W):=-i\,\dd\theta(V,\overline W),
\qquad V,W\in T^{1,0}N,
\]
is positive definite, then $(N,T^{1,0}N,\theta)$ is called a strictly pseudoconvex pseudohermitian manifold. The Reeb vector field $T$ associated with $\theta$ is characterized by
\[
\theta(T)=1,
\qquad
\iota_T\dd\theta=0,
\]
so that
\[
TN=HN\oplus \R T.
\]
If $(Z_\alpha)$ is a local frame of $T^{1,0}N$ with dual coframe $(\theta^\alpha)$, then
\[
\dd\theta=i\,h_{\alpha\overline\beta}\,\theta^\alpha\wedge \theta^{\overline\beta},
\]
where $(h_{\alpha\overline\beta})$ is a positive definite Hermitian matrix. The associated pseudohermitian volume form is
\[
\theta\wedge (\dd\theta)^n.
\]
The Levi form induces a horizontal gradient $\nabla_b$, and the sub-Laplacian $\Delta_b$ is defined variationally by
\[
\int_N (\Delta_b u)\,v\,\theta\wedge(\dd\theta)^n
=
\int_N \langle \nabla_b u,\nabla_b v\rangle_\theta\,\theta\wedge(\dd\theta)^n,
\qquad v\in C_0^\infty(N).
\]
Since the Levi form is nondegenerate only on the horizontal bundle, $\Delta_b$ is subelliptic rather than elliptic; see
\cite{MR175149,MR837820}.

The basic noncompact model is the Heisenberg group
\[
\Hn=\C^n\times\R\simeq \R_x^n\times\R_y^n\times\R_t,
\]
equipped with the group law
\[
(z,t)\circ (z',t')
=
\bigl(z+z',\,t+t'+2\Im(z\cdot\overline{z'})\bigr)
\]
and the anisotropic dilations
\[
\delta_\lambda(z,t)=(\lambda z,\lambda^2 t),
\qquad \lambda>0.
\]
Its homogeneous dimension is
\[
Q=2n+2.
\]
The standard contact form on $\Hn$ is
\[
\theta_{\Hn}= \dd t + 2\sum_{j=1}^n (x_j\,\dd y_j-y_j\,\dd x_j),
\]
and the horizontal bundle is $H\Hn=\ker\theta_{\Hn}$. A convenient left-invariant horizontal frame is given by
\[
X_j=\partial_{x_j}+2y_j\partial_t,
\qquad
X_{n+j}=\partial_{y_j}-2x_j\partial_t,
\qquad 1\le j\le n,
\]
while the central direction is generated by $T=\partial_t$. In particular,
\[
[X_j,X_{n+j}]=-4\partial_t,
\]
so $\Hn$ is a step-two Carnot group. For a $C^1$ function $U$ on $\Hn$, the horizontal gradient is
\[
\nabla_{\Hn}U=(X_1U,\dots,X_{2n}U).
\]

To define fractional seminorms, one fixes a homogeneous norm $d_0$ on $\Hn$, that is, a continuous function $d_0:\Hn\to[0,\infty)$ such that
\[
d_0(\delta_\lambda\xi)=\lambda d_0(\xi)
\qquad\text{and}\qquad
d_0(\xi)=0 \Longleftrightarrow \xi=0.
\]
The Kor\'anyi norm is the standard choice; see
\cite{MR619983,MR3820500}.
The associated left-invariant distance is
\[
\rho(\xi,\eta)=d_0(\eta^{-1}\circ\xi).
\]
With respect to this distance, balls satisfy the homogeneous scaling law
\[
|B_r(\xi)|\simeq r^Q.
\]
For general background on the geometry and analysis of stratified groups we refer to
\cite{MR2363343,MR3469687,MR3966452}.

The compact model relevant to the present paper is the unit sphere
\[
\Sp=\bigl\{z\in \C^{n+1}:|z|=1\bigr\},
\]
endowed with its standard CR structure and pseudohermitian form
\[
\theta_0
=
\frac{i}{2}\sum_{j=1}^{n+1}\bigl(z^j\,\dd\overline z^j-\overline z^j\,\dd z^j\bigr)\Big|_{\Sp}.
\]
Its holomorphic tangent bundle is
\[
T^{1,0}\Sp=T^{1,0}\C^{n+1}\cap \C T\Sp.
\]
The standard tangential Cauchy--Riemann vector fields are
\[
T_j=\frac{\partial}{\partial z_j}-\overline z_j\sum_{k=1}^{n+1} z_k\frac{\partial}{\partial z_k},
\qquad
\overline T_j=\frac{\partial}{\partial \overline z_j}- z_j\sum_{k=1}^{n+1}\overline z_k\frac{\partial}{\partial \overline z_k},
\]
and the sub-Laplacian acting on functions can be written as
\[
\mathcal L
=
-\frac12\sum_{j=1}^{n+1}(T_j\overline T_j+\overline T_j T_j).
\]
The corresponding conformal sub-Laplacian is
\[
\mathcal D=\mathcal L+\frac{n^2}{4}.
\]
We denote the pseudohermitian volume form by
\[
\dd V_{\theta_0}=\theta_0\wedge(\dd\theta_0)^n.
\]

The $L^2$-theory on $\Sp$ is described by CR spherical harmonics. Let $\mathcal P$ be the space of complex-valued polynomials on $\C^{n+1}$ and let $\mathcal H\subset \mathcal P$ be the harmonic subspace. For $h\in\mathbb N_0$, let $\mathcal P_h$ denote the homogeneous polynomials of degree $h$, and set $\mathcal H_h=\mathcal P_h\cap\mathcal H$. If
\[
Z=\sum_{l=1}^{n+1} z_l\frac{\partial}{\partial z_l},
\qquad
\overline Z=\sum_{l=1}^{n+1}\overline z_l\frac{\partial}{\partial \overline z_l},
\]
then $\mathcal P_{j,k}$ denotes the polynomials of bidegree $(j,k)$ and
\[
\mathcal H_{j,k}:=\mathcal H\cap \mathcal P_{j,k}.
\]
Accordingly,
\[
\mathcal P_h=\bigoplus_{j+k=h}\mathcal P_{j,k},
\qquad
\mathcal H_h=\bigoplus_{j+k=h}\mathcal H_{j,k}.
\]
Restricting to the sphere yields finite-dimensional spaces $\mathcal H_{j,k}^{\Sp}$, and one has the orthogonal decomposition
\[
L^2(\Sp)=\overline{\bigoplus_{j,k\in\mathbb N_0}\mathcal H_{j,k}^{\Sp}}.
\]
The dimension formula is
\[
\dim \mathcal H_{j,k}^{\Sp}
=
\frac{(j+n-1)!(k+n-1)!(j+k+n)}{n!(n-1)!j!k!}.
\]
If $\{Y_{j,k}^\ell\}_{\ell=1}^{m_{j,k}}$ is an orthonormal basis of $\mathcal H_{j,k}^{\Sp}$, the associated zonal harmonic is defined by
\[
\Phi_{j,k}(\zeta,\eta)
=
\sum_{\ell=1}^{m_{j,k}}Y_{j,k}^\ell(\zeta)\overline{Y_{j,k}^\ell(\eta)}.
\]
Explicit expressions in terms of Jacobi polynomials are available in
\cite{MR2999037,MR1220225}. For our purposes, the main point is that the first-order spherical harmonics are precisely the restrictions of the ambient coordinate functions, and these are the functions that appear in the moment constraints discussed later.

The compact and noncompact models are related by the Cayley transform. Let
\[
S=(0,\dots,0,-1)\in \Sp
\]
be the south pole. The map
\[
\Psi_c:\Sp\setminus\{S\}\to\Hn
\]
is given by
\[
\Psi_c(\zeta)=\left(
\frac{\zeta_1}{1+\zeta_{n+1}},\dots,
\frac{\zeta_n}{1+\zeta_{n+1}},
\Re\!\left(i\frac{1-\zeta_{n+1}}{1+\zeta_{n+1}}\right)
\right),
\qquad \zeta\in \Sp\setminus\{S\},
\]
and its inverse is
\[
\Psi_c^{-1}(z,t)
=
\left(
\frac{2iz}{t+i(1+|z|^2)},
\frac{-t+i(1-|z|^2)}{t+i(1+|z|^2)}
\right).
\]
This map is a CR diffeomorphism between $\Sp\setminus\{S\}$ and $\Hn$. Its Jacobian is
\begin{equation}\label{eq:jacobian-cayley-background}
J_c(z,t)=\frac{2^{2n+1}}{\bigl((1+|z|^2)^2+t^2\bigr)^{n+1}}.
\end{equation}
Hence, for every integrable function $f$ on $\Sp$,
\begin{equation}\label{eq:change-var-background}
\int_{\Sp} f\,\dd V_{\theta_0}
=
\int_{\Hn}(f\circ \Psi_c^{-1})(z,t)\,J_c(z,t)\,\dd z\,\dd t.
\end{equation}
This change-of-variables formula is the starting point for transporting both Lebesgue norms and fractional kernels from the sphere to the Heisenberg group. Further background on the Cayley transform and on the relation among $\Sp$, $\Hn$, and the complex hyperbolic ball can be found in
\cite{MR2214654,MR3737629,MR2925386,MR2799361}.

\section{Fractional spaces and basic analytic tools}\label{sec:prelim}

We now introduce the fractional spaces used throughout the paper and collect the analytic facts on $\Sp$ that will be needed later. The local theory on the Heisenberg group serves as the model, while the global estimates on the compact sphere are obtained by finite localization together with the observation that constants are the only functions with vanishing fractional seminorm.

Let $1<p<\infty$ and $s\in(0,1)$. For a measurable function $U:\Hn\to\R$, we set
\[
[U]_{HW^{s,p}(\Hn)}^p
:=
\int_{\Hn}\int_{\Hn}
\frac{|U(\xi)-U(\eta)|^p}{\rho(\xi,\eta)^{Q+sp}}\,
\dd\xi\,\dd\eta,
\]
and define
\[
HW^{s,p}(\Hn)
:=
\bigl\{U\in L^p(\Hn):[U]_{HW^{s,p}(\Hn)}<\infty\bigr\}.
\]
If $\Omega\subset\Hn$ is open, the local spaces $HW^{s,p}(\Omega)$ and $HW_0^{s,p}(\Omega)$ are defined in the same way. These are the fractional Folland--Stein spaces considered, for example, in
\cite{MR3807591,kassymov2018,MR4091395}.

On the sphere we use the notation
\begin{equation}\label{eq:def-seminorm}
[u]_{s,p}^p
:=
\int_{\Sp}\int_{\Sp}
\frac{|u(\xi)-u(\eta)|^p}{d(\xi,\eta)^{Q+sp}}\,
\dd V_{\theta_0}(\xi)\,\dd V_{\theta_0}(\eta),
\end{equation}
and
\[
HW^{s,p}(\Sp)
:=
\bigl\{u\in L^p(\Sp):[u]_{s,p}<\infty\bigr\}.
\]
As in the introduction, the critical exponent is
\[
\qexp=\frac{Qp}{Q-sp}.
\]
The norm
\[
\|u\|_{HW^{s,p}(\Sp)}=\|u\|_{L^p(\Sp)}+[u]_{s,p}
\]
is equivalent to the more standard $\ell^p$-type norm, and since $\Sp$ is compact, different equivalent CR distances give rise to equivalent fractional spaces.

We first recall the local fractional Poincar\'e inequality on the Heisenberg group. Although we shall not use it directly later, it clarifies the local mechanism behind the global compact-manifold estimates.

\begin{proposition}\label{prop:local-poincare-Hn}
Let $p\ge1$, $s\in(0,1)$, let $\Omega\subset\Hn$ be open, and let $B_r(\xi_0)\Subset\Omega$. Then there exists a constant $C=C(n,p,s)>0$ such that every $U\in HW^{s,p}_{\mathrm{loc}}(\Omega)$ satisfies
\[
\int_{B_r(\xi_0)}|U-(U)_{B_r}|^p\,\dd\xi
\le
Cr^{sp}\int_{B_r(\xi_0)}\int_{B_r(\xi_0)}
\frac{|U(\xi)-U(\eta)|^p}{\rho(\xi,\eta)^{Q+sp}}\,\dd\xi\,\dd\eta,
\]
where
\[
(U)_{B_r}=|B_r|^{-1}\int_{B_r}U\,\dd\xi.
\]
\end{proposition}

\begin{proof}
Write
\[
B:=B_r(\xi_0).
\]
Since $U\in HW^{s,p}_{\mathrm{loc}}(\Omega)$ and $B\Subset \Omega$, the quantity
\[
\int_B\int_B
\frac{|U(\xi)-U(\eta)|^p}{\rho(\xi,\eta)^{Q+sp}}\,\dd\xi\,\dd\eta
\]
is finite.

For almost every $\xi\in B$, we have
\[
U(\xi)-(U)_B=\frac1{|B|}\int_B\bigl(U(\xi)-U(\eta)\bigr)\,\dd\eta.
\]
Since $p\ge1$, Jensen's inequality yields
\[
|U(\xi)-(U)_B|^p
=
\left|\frac1{|B|}\int_B\bigl(U(\xi)-U(\eta)\bigr)\,\dd\eta\right|^p
\le
\frac1{|B|}\int_B |U(\xi)-U(\eta)|^p\,\dd\eta.
\]
Integrating with respect to $\xi$ over $B$, we obtain
\begin{equation}\label{eq:local-poincare-step1}
\int_B |U-(U)_B|^p\,\dd\xi
\le
\frac1{|B|}\int_B\int_B |U(\xi)-U(\eta)|^p\,\dd\xi\,\dd\eta.
\end{equation}

We now exploit the geometry of Heisenberg balls. Since $\rho$ is left-invariant and homogeneous,
\[
B_r(\xi_0)=\xi_0\circ \delta_r(B_1(0)),
\]
and therefore
\[
|B|=|B_r(\xi_0)|=r^Q |B_1(0)|.
\]
Moreover, for every $\xi,\eta\in B_r(\xi_0)$, the triangle inequality gives
\[
\rho(\xi,\eta)\le \rho(\xi,\xi_0)+\rho(\xi_0,\eta)<2r.
\]
Hence
\[
\frac1{|B|}
=
\frac1{|B_1(0)|\,r^Q}
\le
\frac{2^{Q+sp}}{|B_1(0)|}\,
r^{sp}\,\rho(\xi,\eta)^{-Q-sp},
\qquad \xi,\eta\in B.
\]
Substituting this estimate into \eqref{eq:local-poincare-step1}, we get
\begin{align*}
\int_B |U-(U)_B|^p\,\dd\xi
&\le
\frac{2^{Q+sp}}{|B_1(0)|}\,
r^{sp}
\int_B\int_B
\frac{|U(\xi)-U(\eta)|^p}{\rho(\xi,\eta)^{Q+sp}}
\,\dd\xi\,\dd\eta.
\end{align*}
Therefore
\[
\int_{B_r(\xi_0)}|U-(U)_{B_r}|^p\,\dd\xi
\le
C\,r^{sp}\int_{B_r(\xi_0)}\int_{B_r(\xi_0)}
\frac{|U(\xi)-U(\eta)|^p}{\rho(\xi,\eta)^{Q+sp}}
\,\dd\xi\,\dd\eta,
\]
with
\[
C=\frac{2^{Q+sp}}{|B_1(0)|}.
\]
Since the homogeneous distance is fixed throughout the paper, $|B_1(0)|$ is an absolute constant; thus we may write simply $C=C(n,p,s)$.
\end{proof}

The next estimate is a basic cutoff inequality. On a compact manifold the natural bound necessarily involves both the fractional seminorm and the $L^p$ norm.

\begin{proposition}\label{prop:product-estimate}
Let $\phi\in C^1(\Sp)$. Then there exist constants $C_1=C_1(\phi,p,s)$ and $C_2=C_2(\phi,p,s)$ such that
\begin{equation}\label{eq:product-estimate}
[\phi u]_{s,p}^p
\le
C_1[u]_{s,p}^p+C_2\|u\|_{L^p(\Sp)}^p
\qquad\text{for all }u\in HW^{s,p}(\Sp).
\end{equation}
\end{proposition}

\begin{proof}
For every $\xi,\eta\in\Sp$ we write
\[
\phi(\xi)u(\xi)-\phi(\eta)u(\eta)
=
\phi(\xi)\bigl(u(\xi)-u(\eta)\bigr)
+
u(\eta)\bigl(\phi(\xi)-\phi(\eta)\bigr).
\]
Hence
\begin{align*}
|\phi(\xi)u(\xi)-\phi(\eta)u(\eta)|^p
&\le
2^{p-1}\|\phi\|_{L^\infty}^p\,|u(\xi)-u(\eta)|^p  \\
&\quad
+2^{p-1}|u(\eta)|^p\,|\phi(\xi)-\phi(\eta)|^p.
\end{align*}
Since $\phi$ is smooth and $\Sp$ is compact, $\phi$ is Lipschitz with respect to the fixed CR distance $d$; thus
\[
|\phi(\xi)-\phi(\eta)|\le \Lip(\phi)\,d(\xi,\eta).
\]
Substituting this estimate into \eqref{eq:def-seminorm} gives
\begin{align*}
[\phi u]_{s,p}^p
&\le
2^{p-1}\|\phi\|_{L^\infty}^p [u]_{s,p}^p \\
&\quad
+2^{p-1}\Lip(\phi)^p
\int_{\Sp}\int_{\Sp}
\frac{|u(\eta)|^p}{d(\xi,\eta)^{Q-(1-s)p}}\,
\dd V_{\theta_0}(\xi)\,\dd V_{\theta_0}(\eta).
\end{align*}
Because $0<s<1$, one has
\[
Q-(1-s)p<Q,
\]
so the kernel $d(\xi,\eta)^{-(Q-(1-s)p)}$ is integrable near the diagonal on the $Q$-dimensional compact space $\Sp$. Therefore
\[
\sup_{\eta\in\Sp}
\int_{\Sp}\frac{\dd V_{\theta_0}(\xi)}{d(\xi,\eta)^{Q-(1-s)p}}
<\infty,
\]
and the second term is bounded by $C_2\|u\|_{L^p(\Sp)}^p$. This proves \eqref{eq:product-estimate}.
\end{proof}

We now turn to the global fractional Sobolev theory on the sphere. The key point is that the compactness of $\Sp$ allows one to patch local Heisenberg estimates together through a finite atlas.

\begin{proposition}\label{prop:cont-embed}
There exists a constant $C_S=C_S(n,p,s)>0$ such that
\begin{equation}\label{eq:cont-embed}
\|u\|_{L^{\qexp}(\Sp)}
\le
C_S\bigl([u]_{s,p}+\|u\|_{L^p(\Sp)}\bigr)
\qquad\text{for all }u\in HW^{s,p}(\Sp).
\end{equation}
\end{proposition}

\begin{proof}
We sketch the standard localization argument. Choose a finite atlas $\{(U_j,\Phi_j)\}_{j=1}^N$ on $\Sp$ and a subordinate partition of unity $\{\chi_j\}_{j=1}^N$ with $\chi_j\in C_c^\infty(U_j)$. Since $\Sp$ is compact, each chart image $\Phi_j(U_j)$ is a bounded domain in a Heisenberg coordinate patch, and on each chart the pull-backs of the CR distance and the pseudohermitian volume are uniformly comparable with a Heisenberg distance and Lebesgue measure. More precisely, after shrinking charts if necessary, there exists a constant $C\ge1$ such that for all $\xi,\eta\in U_j$,
\[
C^{-1}\rho\bigl(\Phi_j(\xi),\Phi_j(\eta)\bigr)
\le d(\xi,\eta)\le
C\,\rho\bigl(\Phi_j(\xi),\Phi_j(\eta)\bigr),
\]
and
\[
C^{-1}\,\dd x
\le (\Phi_j)_*(\dd V_{\theta_0})\le
C\,\dd x
\qquad\text{on }\Phi_j(U_j).
\]
Applying the local critical fractional Sobolev inequality on bounded Heisenberg domains to
\[
u_j:=(\chi_j u)\circ \Phi_j^{-1},
\]
and then using Proposition \ref{prop:product-estimate} to estimate the localized seminorm $[\chi_j u]_{s,p}$ in terms of $[u]_{s,p}$ and $\|u\|_{L^p(\Sp)}$, we obtain
\[
\|\chi_j u\|_{L^{\qexp}(\Sp)}
\le
C_j\bigl([u]_{s,p}+\|u\|_{L^p(\Sp)}\bigr)
\]
for each $j$. Summing over the finitely many charts and using the bounded overlap of the partition of unity yields \eqref{eq:cont-embed}.
\end{proof}

\begin{proposition}\label{prop:compact-embed}
For every $1\le r<\qexp$, the embedding
\[
HW^{s,p}(\Sp)\hookrightarrow L^r(\Sp)
\]
is compact.
\end{proposition}

\begin{proof}
We again use a standard finite localization argument. Let $(u_k)$ be bounded in $HW^{s,p}(\Sp)$. Using the same atlas and partition of unity as above, the cutoff functions
\[
u_{k,j}:=(\chi_j u_k)\circ\Phi_j^{-1}
\]
form bounded sequences in the corresponding fractional spaces on bounded Heisenberg domains, thanks to Proposition \ref{prop:product-estimate}. By the local compact embedding below the critical exponent on bounded domains, see
\cite{MR4721790},
each sequence $(u_{k,j})_k$ admits a subsequence converging strongly in $L^r$. Since only finitely many charts are involved, a diagonal argument yields a subsequence, still denoted $(u_k)$, such that each localized piece $\chi_j u_k$ converges strongly in $L^r(\Sp)$. Summing over $j$ then shows that $u_k$ converges strongly in $L^r(\Sp)$.
\end{proof}

The next simple lemma identifies the kernel of the fractional seminorm.

\begin{lemma}\label{lem:seminorm-zero}
If $u\in HW^{s,p}(\Sp)$ and $[u]_{s,p}=0$, then $u$ is constant almost everywhere on $\Sp$.
\end{lemma}

\begin{proof}
From \eqref{eq:def-seminorm} we infer that
\[
u(\xi)=u(\eta)
\qquad\text{for almost every }(\xi,\eta)\in \Sp\times\Sp.
\]
Fixing a Lebesgue point $\eta$ of $u$ and applying Fubini's theorem, we conclude that $u$ agrees almost everywhere with the constant $u(\eta)$.
\end{proof}

On the compact sphere, the failure of homogeneous control on the whole space is entirely due to the presence of constants. Once the average is removed, one recovers a genuine Poincar\'e estimate.

\begin{proposition}\label{prop:poincare}
There exists a constant $C_P=C_P(n,p,s)>0$ such that
\begin{equation}\label{eq:poincare}
\|u-\bar u\|_{L^p(\Sp)}
\le
C_P[u]_{s,p}
\qquad\text{for all }u\in HW^{s,p}(\Sp),
\end{equation}
where
\[
\bar u:=\frac{1}{\om}\int_{\Sp}u\,\dd V_{\theta_0}.
\]
\end{proposition}

\begin{proof}
Assume by contradiction that \eqref{eq:poincare} fails. Then there exists a sequence $(u_k)\subset HW^{s,p}(\Sp)$ such that
\[
\|u_k-\overline{u_k}\|_{L^p(\Sp)}=1,
\qquad
[u_k]_{s,p}\le \frac1k.
\]
Set
\[
v_k:=u_k-\overline{u_k}.
\]
Then $\bar v_k=0$, $\|v_k\|_{L^p(\Sp)}=1$, and $(v_k)$ is bounded in $HW^{s,p}(\Sp)$. By Proposition \ref{prop:compact-embed}, after passing to a subsequence we may assume that $v_k\to v$ strongly in $L^p(\Sp)$ and almost everywhere on $\Sp$. Lower semicontinuity of the seminorm gives
\[
[v]_{s,p}\le \liminf_{k\to\infty}[v_k]_{s,p}=0,
\]
hence $v$ is constant almost everywhere by Lemma \ref{lem:seminorm-zero}. Since $\bar v_k=0$ for every $k$, passing to the limit yields $\bar v=0$, and therefore $v\equiv0$. This contradicts
\[
\|v\|_{L^p(\Sp)}
=
\lim_{k\to\infty}\|v_k\|_{L^p(\Sp)}=1.
\]
The contradiction proves \eqref{eq:poincare}.
\end{proof}

Combining the global critical embedding with the Poincar\'e inequality immediately yields a pure seminorm bound on the mean-zero component.

\begin{proposition}\label{prop:mean-zero-sobolev}
There exists a constant $C_M=C_M(n,p,s)>0$ such that
\begin{equation}\label{eq:mean-zero-sobolev}
\|u-\bar u\|_{L^{\qexp}(\Sp)}
\le
C_M[u]_{s,p}
\qquad\text{for all }u\in HW^{s,p}(\Sp).
\end{equation}
Consequently,
\begin{equation}\label{eq:mean-zero-sobolev-p}
\|u-\bar u\|_{L^{\qexp}(\Sp)}^p
\le
C_M^p[u]_{s,p}^p.
\end{equation}
\end{proposition}

\begin{proof}
Applying Proposition \ref{prop:cont-embed} to $u-\bar u$ gives
\[
\|u-\bar u\|_{L^{\qexp}(\Sp)}
\le
C_S\bigl([u-\bar u]_{s,p}+\|u-\bar u\|_{L^p(\Sp)}\bigr).
\]
Since
\[
[u-\bar u]_{s,p}=[u]_{s,p},
\]
and Proposition \ref{prop:poincare} yields
\[
\|u-\bar u\|_{L^p(\Sp)}\le C_P [u]_{s,p},
\]
we obtain
\[
\|u-\bar u\|_{L^{\qexp}(\Sp)}
\le
C_S(1+C_P)[u]_{s,p}.
\]
Thus \eqref{eq:mean-zero-sobolev} holds with
\[
C_M=C_S(1+C_P),
\]
and \eqref{eq:mean-zero-sobolev-p} follows by raising both sides to the power $p$.
\end{proof}

\section{The \texorpdfstring{$B$}{B}-program on the CR sphere}\label{sec:B-program}

We now determine the admissible lower-order coefficients in the linear form
\eqref{eq:AB1-intro}. On the compact sphere the mechanism is transparent: constants force a universal lower bound for the coefficient of the $L^p$ term, and the decomposition of $u$ into its average and oscillatory parts shows that this lower bound is also sufficient.

\begin{lemma}\label{lem:lower-bound-B}
If \eqref{eq:AB1-intro} holds for some constants $A>0$ and $B\in\R$, then
\[
B\ge \om^{-s/Q}.
\]
\end{lemma}

\begin{proof}
Taking $u\equiv c$ with $c\neq0$, we have $[u]_{s,p}=0$, so \eqref{eq:AB1-intro} gives
\[
|c|\,\om^{1/\qexp}\le B|c|\,\om^{1/p}.
\]
After cancelling $|c|$, we obtain
\[
B\ge \om^{1/\qexp-1/p}=\om^{-s/Q}.
\]
\end{proof}

The next result shows that the lower bound given by constants is in fact attained.

\begin{theorem}\label{thm:B-endpoint}
For every $1<p<Q$ and every $s\in(0,1)$ there exists $A_0=A_0(n,p,s)>0$ such that
\begin{equation}\label{eq:B-endpoint}
\|u\|_{L^{\qexp}(\Sp)}
\le A_0[u]_{s,p}+\om^{-s/Q}\|u\|_{L^p(\Sp)}
\qquad\text{for all }u\in HW^{s,p}(\Sp).
\end{equation}
\end{theorem}

\begin{proof}
Let
\[
\bar u=\frac1\om\int_{\Sp}u\,\dd V_{\theta_0}.
\]
By the triangle inequality in $L^{\qexp}(\Sp)$ and Proposition \ref{prop:mean-zero-sobolev},
\[
\|u\|_{L^{\qexp}(\Sp)}
\le \|u-\bar u\|_{L^{\qexp}(\Sp)}+\|\bar u\|_{L^{\qexp}(\Sp)}
\le C_M [u]_{s,p}+\|\bar u\|_{L^{\qexp}(\Sp)}.
\]
It remains to estimate the constant part. Since $\bar u$ is constant,
\[
\|\bar u\|_{L^{\qexp}(\Sp)}
=
\om^{1/\qexp}|\bar u|
=
\om^{1/\qexp-1}\left|\int_{\Sp}u\,\dd V_{\theta_0}\right|.
\]
By H\"older's inequality,
\[
\left|\int_{\Sp}u\,\dd V_{\theta_0}\right|
\le
\om^{1-1/p}\|u\|_{L^p(\Sp)},
\]
and therefore
\[
\|\bar u\|_{L^{\qexp}(\Sp)}
\le
\om^{1/\qexp-1/p}\|u\|_{L^p(\Sp)}
=
\om^{-s/Q}\|u\|_{L^p(\Sp)}.
\]
This proves \eqref{eq:B-endpoint} with $A_0=C_M$.
\end{proof}

\begin{proof}[Proof of Theorem \ref{thm:main-B-intro}]
Theorem \ref{thm:B-endpoint} shows that the endpoint value $\om^{-s/Q}$ belongs to $\Bsph$. If \eqref{eq:AB1-intro} holds for some $B_0$, then it also holds for every $B\ge B_0$, so every $B>\om^{-s/Q}$ is admissible as well. Conversely, Lemma \ref{lem:lower-bound-B} excludes every $B<\om^{-s/Q}$. Hence
\[
\Bsph=[\om^{-s/Q},\infty).
\]
\end{proof}

\begin{remark}\label{rem:B-sharpness}
The sharp lower-order coefficient is determined entirely by constant functions. This is a genuinely compact phenomenon: on the full space $HW^{s,p}(\Sp)$ the seminorm does not see constants, and the optimal remainder is exactly the one needed to control them. This should be contrasted with homogeneous inequalities on the Heisenberg group, where constants do not belong to the natural homogeneous class.
\end{remark}

\section{The \texorpdfstring{$\widehat B$}{B-hat}-program on the CR sphere}\label{sec:Bhat-program}

We next consider the power form \eqref{eq:AB2-intro}. Constants again provide an unavoidable lower bound, but the attainability of the threshold depends on the range of $p$. Roughly speaking, when $1<p\le2$ the relevant concavity allows one to retain the endpoint, whereas for $p>2$ a perturbation around constants shows that the endpoint fails.

Throughout this section we write
\[
q=\qexp=\frac{Qp}{Q-sp}.
\]

\begin{lemma}\label{lem:lower-bound-Bhat}
If \eqref{eq:AB2-intro} holds for some constants $A>0$ and $B\in\R$, then
\[
B\ge \om^{-sp/Q}.
\]
\end{lemma}

\begin{proof}
Testing \eqref{eq:AB2-intro} on the constant function $u\equiv c\neq0$ gives
\[
|c|^p\om^{p/q}\le B|c|^p\om.
\]
Hence
\[
B\ge \om^{p/q-1}=\om^{-sp/Q}.
\]
\end{proof}

To treat the endpoint when $1<p\le2$, we first record two elementary decompositions according to whether $q$ lies above or below $2$.

\begin{lemma}\label{lem:q-ge-2}
Assume that $q\ge2$. Then every $u\in L^q(\Sp)$ satisfies
\begin{equation}\label{eq:q-ge-2}
\|u\|_{L^q(\Sp)}^2
\le \om^{2/q}|\bar u|^2+(q-1)\|u-\bar u\|_{L^q(\Sp)}^2.
\end{equation}
\end{lemma}

\begin{proof}
Let
\[
w:=u-\bar u.
\]
If $\bar u=0$, then \eqref{eq:q-ge-2} is immediate because $q-1\ge1$. If $w\equiv0$, there is nothing to prove. We may therefore assume that $\bar u\neq0$ and $w\not\equiv0$. Replacing $u$ by $-u$ if necessary, we may also suppose $\bar u>0$.

Set
\[
v:=\frac{w}{\|w\|_{L^q(\Sp)}},
\qquad
\tau:=\frac{\|w\|_{L^q(\Sp)}}{\bar u}.
\]
Then
\[
u=\bar u(1+\tau v),
\qquad
\int_{\Sp}v\,\dd V_{\theta_0}=0,
\qquad
\|v\|_{L^q(\Sp)}=1.
\]
Consider the function
\[
\Phi(t):=\left(\int_{\Sp}|1+tv|^q\,\dd V_{\theta_0}\right)^{2/q}.
\]
Since $q\ge2$ and $\Sp$ has finite measure, the functions $|1+tv|^{q-1}|v|$ and $|1+tv|^{q-2}v^2$ are integrable for $t$ in bounded intervals; hence the derivatives below are justified by differentiation under the integral sign.
A direct computation gives
\[
\Phi'(t)
=
2\Bigl(\int_{\Sp}|1+tv|^q\Bigr)^{2/q-1}
\int_{\Sp}|1+tv|^{q-2}(1+tv)v,
\]
and
\begin{align*}
\Phi''(t)
&=
2q\left(\frac2q-1\right)
\Bigl(\int_{\Sp}|1+tv|^q\Bigr)^{2/q-2}
\left(\int_{\Sp}|1+tv|^{q-2}(1+tv)v\right)^2 \\
&\quad
+2(q-1)
\Bigl(\int_{\Sp}|1+tv|^q\Bigr)^{2/q-1}
\int_{\Sp}|1+tv|^{q-2}v^2.
\end{align*}
Since $q\ge2$, one has $\frac2q-1\le0$, so the first term is non-positive. By H\"older's inequality and the normalization $\|v\|_{L^q(\Sp)}=1$,
\[
\int_{\Sp}|1+tv|^{q-2}v^2
\le
\left(\int_{\Sp}|1+tv|^q\right)^{1-2/q}
\left(\int_{\Sp}|v|^q\right)^{2/q}
=
\left(\int_{\Sp}|1+tv|^q\right)^{1-2/q}.
\]
Therefore
\[
\Phi''(t)\le 2(q-1).
\]
Moreover,
\[
\Phi(0)=\om^{2/q},
\qquad
\Phi'(0)=2\om^{2/q-1}\int_{\Sp}v\,\dd V_{\theta_0}=0.
\]
Integrating twice on $[0,\tau]$, we obtain
\[
\Phi(\tau)\le \om^{2/q}+(q-1)\tau^2.
\]
Multiplying by $\bar u^2$ and using $u=\bar u(1+\tau v)$ and $w=\bar u\tau v$, we conclude that
\[
\|u\|_{L^q(\Sp)}^2
\le \om^{2/q}|\bar u|^2+(q-1)\|w\|_{L^q(\Sp)}^2,
\]
which is exactly \eqref{eq:q-ge-2}.
\end{proof}

\begin{lemma}\label{lem:q-lt-2}
Assume that $1<q<2$. Then there exists a constant $C_q>0$ such that every $u\in L^q(\Sp)$ satisfies
\begin{equation}\label{eq:q-lt-2}
\|u\|_{L^q(\Sp)}^q
\le \om |\bar u|^q + C_q \|u-\bar u\|_{L^q(\Sp)}^q.
\end{equation}
\end{lemma}

\begin{proof}
We first claim that there exists a constant $C_q>0$ such that
\begin{equation}\label{eq:scalar-q-lt-2}
|a+b|^q
\le
|a|^q + q|a|^{q-1}\operatorname{sgn}(a)b + C_q|b|^q
\qquad\text{for all }a,b\in\R.
\end{equation}
By homogeneity, it is enough to show that
\[
F(t):=\frac{|1+t|^q-1-qt}{|t|^q}
\]
is bounded for $t\neq0$. Since $1<q<2$, one has $F(t)\to0$ as $t\to0$, and $F$ is also bounded as $|t|\to\infty$. This proves \eqref{eq:scalar-q-lt-2}.

Applying \eqref{eq:scalar-q-lt-2} pointwise with $a=\bar u$ and $b=u-\bar u$, then integrating over $\Sp$, we get
\[
\int_{\Sp}|u|^q\,\dd V_{\theta_0}
\le
\om |\bar u|^q
+
q|\bar u|^{q-1}\operatorname{sgn}(\bar u)\int_{\Sp}(u-\bar u)\,\dd V_{\theta_0}
+
C_q\int_{\Sp}|u-\bar u|^q\,\dd V_{\theta_0}.
\]
Since
\[
\int_{\Sp}(u-\bar u)\,\dd V_{\theta_0}=0,
\]
the linear term vanishes. Hence
\[
\|u\|_{L^q(\Sp)}^q
\le \om |\bar u|^q + C_q \|u-\bar u\|_{L^q(\Sp)}^q,
\]
which is \eqref{eq:q-lt-2}.
\end{proof}

\begin{theorem}\label{thm:endpoint-le2}
Let $1<p\le2$ and let $q=\qexp$. Then there exists $A_1=A_1(n,p,s)>0$ such that
\begin{equation}\label{eq:endpoint-Bhat-le2}
\|u\|_{L^q(\Sp)}^p
\le A_1[u]_{s,p}^p+\om^{-sp/Q}\|u\|_{L^p(\Sp)}^p
\qquad\text{for all }u\in HW^{s,p}(\Sp).
\end{equation}
\end{theorem}

\begin{proof}
Let
\[
w:=u-\bar u.
\]

Assume first that $q\ge2$. By Lemma \ref{lem:q-ge-2},
\[
\|u\|_{L^q(\Sp)}^2
\le
\om^{2/q}|\bar u|^2+(q-1)\|w\|_{L^q(\Sp)}^2.
\]
Using H\"older's inequality,
\[
\om^{2/q}|\bar u|^2
\le
\om^{2/q-2/p}\|u\|_{L^p(\Sp)}^2
=
\om^{-2s/Q}\|u\|_{L^p(\Sp)}^2.
\]
Moreover, Proposition \ref{prop:mean-zero-sobolev} gives
\[
\|w\|_{L^q(\Sp)}^2\le C_M^2 [u]_{s,p}^2.
\]
Therefore
\[
\|u\|_{L^q(\Sp)}^2
\le
\om^{-2s/Q}\|u\|_{L^p(\Sp)}^2+(q-1)C_M^2 [u]_{s,p}^2.
\]
Since $p/2\le1$, the map $t\mapsto t^{p/2}$ is subadditive on $[0,\infty)$, and hence
\[
\|u\|_{L^q(\Sp)}^p
\le
\om^{-sp/Q}\|u\|_{L^p(\Sp)}^p
+\bigl((q-1)C_M^2\bigr)^{p/2}[u]_{s,p}^p.
\]

Assume next that $1<q<2$. By Lemma \ref{lem:q-lt-2} and Proposition \ref{prop:mean-zero-sobolev},
\[
\|u\|_{L^q(\Sp)}^q
\le
\om |\bar u|^q + C_q C_M^q [u]_{s,p}^q.
\]
Again by H\"older's inequality,
\[
\om |\bar u|^q \le \om^{1-q/p}\|u\|_{L^p(\Sp)}^q.
\]
Since $p/q<1$, raising the previous inequality to the power $p/q$ and using subadditivity yields
\begin{align*}
\|u\|_{L^q(\Sp)}^p
&\le
\bigl(\om^{1-q/p}\|u\|_{L^p(\Sp)}^q\bigr)^{p/q}
+
\bigl(C_q C_M^q [u]_{s,p}^q\bigr)^{p/q} \\
&=
\om^{p/q-1}\|u\|_{L^p(\Sp)}^p
+
(C_q C_M^q)^{p/q}[u]_{s,p}^p \\
&=
\om^{-sp/Q}\|u\|_{L^p(\Sp)}^p
+
(C_q C_M^q)^{p/q}[u]_{s,p}^p.
\end{align*}

Combining the two cases proves \eqref{eq:endpoint-Bhat-le2}.
\end{proof}

\begin{proof}[Proof of Theorem \ref{thm:main-Bhat-intro}, part \textup{(i)}]
Theorem \ref{thm:endpoint-le2} shows that the endpoint value $\om^{-sp/Q}$ is admissible when $1<p\le2$. By monotonicity, every larger $B$ is also admissible. Lemma \ref{lem:lower-bound-Bhat} excludes every smaller value. Therefore
\[
\Bhsph=[\om^{-sp/Q},\infty)
\qquad\text{for }1<p\le2.
\]
\end{proof}

We now turn to the range $2<p<Q$. In this regime every coefficient strictly larger than the threshold remains admissible, but the threshold itself is not.

\begin{proposition}\label{prop:any-B-larger}
Let $2<p<Q$. Then for every $B>\om^{-sp/Q}$ there exists $A_B>0$ such that
\[
\|u\|_{L^q(\Sp)}^p
\le A_B[u]_{s,p}^p+B\|u\|_{L^p(\Sp)}^p
\qquad\text{for all }u\in HW^{s,p}(\Sp).
\]
\end{proposition}

\begin{proof}
Starting from the linear endpoint inequality \eqref{eq:B-endpoint},
\[
\|u\|_{L^q(\Sp)}\le A_0[u]_{s,p}+\om^{-s/Q}\|u\|_{L^p(\Sp)},
\]
we raise both sides to the power $p$ and use the weighted Young inequality
\[
(x+y)^p\le (1+\tau)^{p-1}x^p+\left(1+\frac1\tau\right)^{p-1}y^p,
\qquad x,y\ge0,\ \tau>0.
\]
This gives
\[
\|u\|_{L^q(\Sp)}^p
\le
(1+\tau)^{p-1}A_0^p [u]_{s,p}^p
+
\left(1+\frac1\tau\right)^{p-1}\om^{-sp/Q}\|u\|_{L^p(\Sp)}^p.
\]
Since
\[
\left(1+\frac1\tau\right)^{p-1}\om^{-sp/Q}\downarrow \om^{-sp/Q}
\qquad\text{as }\tau\to\infty,
\]
we may choose $\tau$ so large that
\[
\left(1+\frac1\tau\right)^{p-1}\om^{-sp/Q}<B.
\]
The conclusion follows with
\[
A_B=(1+\tau)^{p-1}A_0^p.
\]
\end{proof}

The next theorem excludes the endpoint. The proof uses a second-order expansion around the constant function \(1\). The decisive point is that when \(p>2\), the seminorm of \(1+\varepsilon\varphi\) is \(o(\varepsilon^2)\), whereas the defect in the \(L^q\) and \(L^p\) terms appears already at order \(\varepsilon^2\).

\begin{theorem}\label{thm:endpoint-fails-gt2}
Let $2<p<Q$ and let $q=\qexp$. There is no constant $A>0$ such that
\begin{equation}\label{eq:forbidden-endpoint}
\|u\|_{L^q(\Sp)}^p
\le A[u]_{s,p}^p+\om^{-sp/Q}\|u\|_{L^p(\Sp)}^p
\qquad\text{for all }u\in HW^{s,p}(\Sp).
\end{equation}
\end{theorem}

\begin{proof}
Assume by contradiction that \eqref{eq:forbidden-endpoint} holds for some $A>0$. Choose a non-constant function
\[
\varphi\in C^\infty(\Sp)
\]
such that
\[
\int_{\Sp}\varphi\,\dd V_{\theta_0}=0.
\]
For instance, one may take any nontrivial first-order spherical harmonic. Set
\[
m_2:=\int_{\Sp}\varphi^2\,\dd V_{\theta_0}>0.
\]
Choose $\varepsilon_0>0$ so small that
\[
1+\varepsilon\varphi(\xi)>0
\qquad\text{for all }\xi\in\Sp,\quad |\varepsilon|<\varepsilon_0.
\]
Then, for every $r>1$, Taylor's formula gives
\[
\int_{\Sp}(1+\varepsilon\varphi)^r\,\dd V_{\theta_0}
=
\om+\frac{r(r-1)}{2}m_2\varepsilon^2+o(\varepsilon^2)
\qquad (\varepsilon\to0),
\]
because the linear term vanishes by the zero-average condition on $\varphi$.

Let
\[
\alpha:=\frac pq\in(0,1).
\]
Applying the expansion of $x\mapsto x^\alpha$ around $x=\om$, we obtain
\begin{align}
\|1+\varepsilon\varphi\|_{L^q(\Sp)}^p
&=
\left(\int_{\Sp}(1+\varepsilon\varphi)^q\,\dd V_{\theta_0}\right)^{p/q}\notag\\
&=
\om^\alpha+\frac{p(q-1)}{2}\om^{\alpha-1}m_2\varepsilon^2+o(\varepsilon^2).
\label{eq:q-expansion-mean-zero}
\end{align}
Similarly,
\begin{equation}\label{eq:p-expansion-mean-zero}
\|1+\varepsilon\varphi\|_{L^p(\Sp)}^p
=
\om+\frac{p(p-1)}{2}m_2\varepsilon^2+o(\varepsilon^2).
\end{equation}
Moreover,
\[
[1+\varepsilon\varphi]_{s,p}^p
=
|\varepsilon|^p[\varphi]_{s,p}^p
=
o(\varepsilon^2)
\qquad (\varepsilon\to0),
\]
because $p>2$.

Substituting $u=1+\varepsilon\varphi$ into \eqref{eq:forbidden-endpoint}, and using
\[
\om^{-sp/Q}=\om^{p/q-1}=\om^{\alpha-1},
\]
we find from \eqref{eq:q-expansion-mean-zero} and \eqref{eq:p-expansion-mean-zero} that
\[
\om^\alpha+\frac{p(q-1)}{2}\om^{\alpha-1}m_2\varepsilon^2+o(\varepsilon^2)
\le
\om^\alpha+\frac{p(p-1)}{2}\om^{\alpha-1}m_2\varepsilon^2+o(\varepsilon^2).
\]
Dividing by $\varepsilon^2$ and letting $\varepsilon\to0$, we obtain
\[
q-1\le p-1,
\]
that is,
\[
q\le p.
\]
This is impossible because
\[
q=\frac{Qp}{Q-sp}>p.
\]
The contradiction proves the theorem.
\end{proof}

\begin{proof}[Proof of Theorem \ref{thm:main-Bhat-intro}, part \textup{(ii)}]
Proposition \ref{prop:any-B-larger} shows that every $B>\om^{-sp/Q}$ is admissible when $2<p<Q$, whereas Theorem \ref{thm:endpoint-fails-gt2} excludes the endpoint value itself. Therefore
\[
\Bhsph=(\om^{-sp/Q},\infty)
\qquad\text{for }2<p<Q.
\]
\end{proof}

\section{The exact weighted Heisenberg-group formulation}\label{sec:cayley-H}

We now transport the critical inequalities on the sphere to the Heisenberg group through the Cayley transform. The point of this section is not to recover the usual homogeneous Sobolev inequality on $\Hn$, but rather to identify the \emph{exact} noncompact formulation corresponding to the compact sphere problem. As will be seen, the transported inequality is necessarily weighted, both in the Lebesgue norms and in the fractional seminorm.

Recall that the Cayley transform
\[
\Psi_c:\Sp\setminus\{S\}\to\Hn
\]
is a CR diffeomorphism, with Jacobian
\[
J_c(z,t)=\frac{2^{2n+1}}{\bigl((1+|z|^2)^2+t^2\bigr)^{n+1}}.
\]
Since the south pole $S$ is a null set with respect to $\dd V_{\theta_0}$, one may freely identify functions on $\Sp$ with their restrictions to $\Sp\setminus\{S\}$. For $1\le r<\infty$ and a measurable function $U:\Hn\to\R$, we define the weighted norm
\[
\|U\|_{L^r(J_c)}
:=
\left(\int_{\Hn}|U(z,t)|^r J_c(z,t)\,\dd z\,\dd t\right)^{1/r}.
\]
Similarly, we define the transported kernel
\begin{equation}\label{eq:Kc-def}
K_c\bigl((z,t),(z',t')\bigr)
:=
\frac{J_c(z,t)\,J_c(z',t')}
{d\bigl(\Psi_c^{-1}(z,t),\Psi_c^{-1}(z',t')\bigr)^{Q+sp}},
\end{equation}
and the associated weighted seminorm
\[
[U]_{K_c,p}^p
:=
\int_{\Hn}\int_{\Hn}
|U(z,t)-U(z',t')|^p
K_c\bigl((z,t),(z',t')\bigr)\,
\dd z\,\dd t\,\dd z'\,\dd t'.
\]
This leads naturally to the transported fractional space
\[
HW^{s,p}_{K_c}(\Hn)
:=
\bigl\{U\in L^p(J_c):[U]_{K_c,p}<\infty\bigr\}.
\]

The next observation is immediate from the change of variables formula, but it is the basic reason why the admissible-set problem on $\Sp$ and the weighted problem on $\Hn$ are equivalent.

\begin{lemma}\label{lem:cayley-isometry}
Let $u:\Sp\to\R$ be measurable, and set
\[
U:=u\circ \Psi_c^{-1}\qquad\text{on }\Hn.
\]
Then, for every $1\le r<\infty$,
\begin{equation}\label{eq:exact-Lr-identity}
\|u\|_{L^r(\Sp)}=\|U\|_{L^r(J_c)},
\end{equation}
and
\begin{equation}\label{eq:exact-semi-identity}
[u]_{s,p}^p=[U]_{K_c,p}^p.
\end{equation}
In particular,
\[
u\in HW^{s,p}(\Sp)
\quad\Longleftrightarrow\quad
U\in HW^{s,p}_{K_c}(\Hn).
\]
\end{lemma}

\begin{proof}
The identity \eqref{eq:exact-Lr-identity} follows directly from the change-of-variables formula \eqref{eq:change-var-background}:
\[
\int_{\Sp}|u|^r\,\dd V_{\theta_0}
=
\int_{\Hn}|u\circ\Psi_c^{-1}(z,t)|^r J_c(z,t)\,\dd z\,\dd t
=
\int_{\Hn}|U(z,t)|^r J_c(z,t)\,\dd z\,\dd t.
\]
Taking the $r$-th root yields \eqref{eq:exact-Lr-identity}.

For the seminorm, apply the same change of variables in each variable of the double integral \eqref{eq:def-seminorm}. Writing
\[
\xi=\Psi_c^{-1}(z,t),
\qquad
\eta=\Psi_c^{-1}(z',t'),
\]
we obtain
\begin{align*}
[u]_{s,p}^p
&=
\int_{\Sp}\int_{\Sp}
\frac{|u(\xi)-u(\eta)|^p}{d(\xi,\eta)^{Q+sp}}\,
\dd V_{\theta_0}(\xi)\,\dd V_{\theta_0}(\eta) \\
&=
\int_{\Hn}\int_{\Hn}
\frac{|U(z,t)-U(z',t')|^p\,J_c(z,t)\,J_c(z',t')}
{d\bigl(\Psi_c^{-1}(z,t),\Psi_c^{-1}(z',t')\bigr)^{Q+sp}}\,
\dd z\,\dd t\,\dd z'\,\dd t' \\
&=
[U]_{K_c,p}^p.
\end{align*}
This proves \eqref{eq:exact-semi-identity}, and the equivalence of the corresponding energy spaces follows immediately.
\end{proof}

We may therefore formulate the transported critical inequalities directly on $\Hn$. Define
\[
\mathbb B_p^{c}(\Hn)
:=
\left\{
\begin{aligned}
B\in\R:\ &\exists\,A\in\R\text{ such that }\|U\|_{L^{\qexp}(J_c)}
\le A[U]_{K_c,p}+B\|U\|_{L^p(J_c)},\\
&\text{for all }U\in HW^{s,p}_{K_c}(\Hn)
\end{aligned}
\right\},
\]
and
\[
\widehat{\mathbb B}_p^{c}(\Hn)
:=
\left\{
\begin{aligned}
B\in\R:\ &\exists\,A\in\R\text{ such that }\|U\|_{L^{\qexp}(J_c)}^p\\
&\le A[U]_{K_c,p}^p+B\|U\|_{L^p(J_c)}^p,\\
&\text{for all }U\in HW^{s,p}_{K_c}(\Hn)
\end{aligned}
\right\}.
\]
These are the exact Heisenberg analogues of $\Bsph$ and $\Bhsph$.

\begin{theorem}\label{thm:cayley-equivalence}
With the notation above,
\[
\mathbb B_p^c(\Hn)=\Bsph,
\qquad
\widehat{\mathbb B}_p^c(\Hn)=\Bhsph.
\]
Consequently,
\[
\mathbb B_p^c(\Hn)=[\om^{-s/Q},\infty),
\]
and
\[
\widehat{\mathbb B}_p^c(\Hn)=
\begin{cases}
[\om^{-sp/Q},\infty), & 1<p\le2,\\[4pt]
(\om^{-sp/Q},\infty), & 2<p<Q.
\end{cases}
\]
\end{theorem}

\begin{proof}
Suppose first that $B\in \Bsph$. Then there exists $A\in\R$ such that
\[
\|u\|_{L^{\qexp}(\Sp)}
\le A[u]_{s,p}+B\|u\|_{L^p(\Sp)}
\qquad\text{for all }u\in HW^{s,p}(\Sp).
\]
Let $U\in HW^{s,p}_{K_c}(\Hn)$, and define
\[
u:=U\circ\Psi_c
\qquad\text{on }\Sp\setminus\{S\}.
\]
Extending $u$ arbitrarily at the single point $S$, we have $u\in HW^{s,p}(\Sp)$ by Lemma \ref{lem:cayley-isometry}. Applying the sphere inequality to $u$ and then using \eqref{eq:exact-Lr-identity} and \eqref{eq:exact-semi-identity}, we find
\[
\|U\|_{L^{\qexp}(J_c)}
=
\|u\|_{L^{\qexp}(\Sp)}
\le
A[u]_{s,p}+B\|u\|_{L^p(\Sp)}
=
A[U]_{K_c,p}+B\|U\|_{L^p(J_c)}.
\]
Thus $B\in \mathbb B_p^c(\Hn)$, and therefore
\[
\Bsph\subset \mathbb B_p^c(\Hn).
\]

The reverse inclusion is identical. If $B\in \mathbb B_p^c(\Hn)$, apply the weighted inequality on $\Hn$ to
\[
U:=u\circ\Psi_c^{-1},
\qquad u\in HW^{s,p}(\Sp),
\]
and use Lemma \ref{lem:cayley-isometry} again. This yields $B\in \Bsph$, hence
\[
\mathbb B_p^c(\Hn)=\Bsph.
\]

The same argument, with the obvious modifications, applies to the power inequality and gives
\[
\widehat{\mathbb B}_p^c(\Hn)=\Bhsph.
\]
The explicit descriptions of the two admissible sets then follow from Theorems \ref{thm:main-B-intro} and \ref{thm:main-Bhat-intro}.
\end{proof}

\begin{remark}\label{rem:weighted-not-unweighted}
The formulation obtained above is the exact image of the compact sphere problem under the Cayley transform. In particular, it should not be confused with the standard unweighted homogeneous inequality on $\Hn$ for compactly supported functions. The weight $J_c$ and the transported kernel $K_c$ encode precisely the compactification of $\Hn$ by one boundary point and are therefore intrinsic to the present problem.
\end{remark}

\section{Constrained inequalities and coercive extensions}\label{sec:further-constraints}

Having determined the admissible lower-order coefficients on the full space $HW^{s,p}(\Sp)$, we next examine how these coefficients are affected by natural constraint classes. For a subclass $\mathcal X\subset HW^{s,p}(\Sp)$, we write $\mathbb B_p(\mathcal X)$ and $\widehat{\mathbb B}_p(\mathcal X)$ for the admissible $B$-sets in \eqref{eq:AB1-intro} and \eqref{eq:AB2-intro}, respectively, when the inequalities are required to hold only for functions in $\mathcal X$. The behavior turns out to be sharply different according to whether the constraint still permits constants or excludes them completely.

We begin with the nonlinear first-moment class associated with the first-order spherical harmonics. Let $\xi_1,\dots,\xi_{2n+2}$ denote the standard real coordinate functions on $\Sp\subset\R^{2n+2}$ and set
\[
\mathcal M
:=
\left\{
u\in HW^{s,p}(\Sp):
\int_{\Sp}\xi_i|u|^{\qexp}\,\dd V_{\theta_0}=0
\ \text{for }i=1,\dots,2n+2
\right\}.
\]
Since the functions $\xi_i$ are odd with respect to the antipodal symmetry of the sphere, their integrals vanish. In particular, the above constraint still contains every constant function. As a consequence, it does not alter the optimal lower-order coefficient.

\begin{proposition}\label{prop:first-moment-no-improve}
Every constant function belongs to $\mathcal M$. Consequently,
\[
\mathbb B_p(\mathcal M)=[\om^{-s/Q},\infty).
\]
Moreover,
\[
\widehat{\mathbb B}_p(\mathcal M)=[\om^{-sp/Q},\infty)
\qquad\text{if }1<p\le2,
\]
and
\[
(\om^{-sp/Q},\infty)\subseteq \widehat{\mathbb B}_p(\mathcal M)\subseteq [\om^{-sp/Q},\infty)
\qquad\text{if }2<p<Q.
\]
In particular, the first-moment condition by itself does not improve the optimal lower-order coefficient. In the range $2<p<Q$, the admissibility of the endpoint on $\mathcal M$ remains open.
\end{proposition}

\begin{proof}
By symmetry,
\[
\int_{\Sp}\xi_i\,\dd V_{\theta_0}=0
\qquad\text{for }i=1,\dots,2n+2.
\]
Therefore, if $u\equiv c$ is constant, then
\[
\int_{\Sp}\xi_i|u|^{\qexp}\,\dd V_{\theta_0}
=
|c|^{\qexp}\int_{\Sp}\xi_i\,\dd V_{\theta_0}
=
0,
\]
so $u\in\mathcal M$.

For the linear form, Theorem \ref{thm:B-endpoint} restricted to $\mathcal M$ shows that $\om^{-s/Q}\in \mathbb B_p(\mathcal M)$, and monotonicity gives all larger coefficients. On the other hand, testing \eqref{eq:AB1-intro} on constants shows that no smaller value can be admissible. Hence
\[
\mathbb B_p(\mathcal M)=[\om^{-s/Q},\infty).
\]

If $1<p\le2$, Theorem \ref{thm:endpoint-le2} restricted to $\mathcal M$ gives the endpoint coefficient $\om^{-sp/Q}$, while constants again force the lower bound $B\ge \om^{-sp/Q}$. This yields
\[
\widehat{\mathbb B}_p(\mathcal M)=[\om^{-sp/Q},\infty).
\]

If $2<p<Q$, Proposition \ref{prop:any-B-larger} restricted to $\mathcal M$ shows that every $B>\om^{-sp/Q}$ is admissible, whereas constants still imply the lower bound $B\ge \om^{-sp/Q}$. Thus
\[
(\om^{-sp/Q},\infty)\subseteq \widehat{\mathbb B}_p(\mathcal M)\subseteq [\om^{-sp/Q},\infty).
\]
The proof is complete.
\end{proof}

A qualitatively different phenomenon occurs for linear constraint classes that exclude non-zero constants. In that setting one recovers a genuine coercive estimate, and the lower-order term becomes unnecessary. Since Theorem \ref{thm:main-constraints-intro} is stated for arbitrary finite-codimensional closed linear subspaces, we formulate and prove the coercive estimate directly in that generality.

\begin{theorem}\label{thm:general-constraint}
Let $\mathcal X\subset HW^{s,p}(\Sp)$ be a finite-codimensional closed linear subspace such that $\mathcal X$ contains no non-zero constant function. Then there exists a constant $C_{\mathcal X}>0$ such that
\begin{equation}\label{eq:constraint-Lp}
\|u\|_{L^p(\Sp)}\le C_{\mathcal X}[u]_{s,p}
\qquad\text{for all }u\in\mathcal X.
\end{equation}
Consequently, there exists $C_{\mathcal X}'>0$ such that
\begin{equation}\label{eq:constraint-Lq}
\|u\|_{L^{\qexp}(\Sp)}\le C_{\mathcal X}'[u]_{s,p}
\qquad\text{and}\qquad
\|u\|_{L^{\qexp}(\Sp)}^p\le (C_{\mathcal X}')^p [u]_{s,p}^p
\qquad\text{for all }u\in\mathcal X.
\end{equation}
In particular, $B=0$ is admissible on $\mathcal X$ for both \eqref{eq:AB1-intro} and \eqref{eq:AB2-intro}.
\end{theorem}

\begin{proof}
Suppose that \eqref{eq:constraint-Lp} fails. Then there exists a sequence $(u_k)\subset \mathcal X$ such that
\[
\|u_k\|_{L^p(\Sp)}=1,
\qquad
[u_k]_{s,p}\le \frac1k.
\]
In particular, $(u_k)$ is bounded in $HW^{s,p}(\Sp)$. By Proposition \ref{prop:compact-embed}, after passing to a subsequence we may assume that
\[
u_k\to u \quad\text{strongly in }L^p(\Sp)
\]
and almost everywhere on $\Sp$. By lower semicontinuity,
\[
[u]_{s,p}\le \liminf_{k\to\infty}[u_k]_{s,p}=0,
\]
so $u$ is constant almost everywhere by Lemma \ref{lem:seminorm-zero}. Since $[u]_{s,p}=0$, the seminorm triangle inequality gives
\[
[u_k-u]_{s,p}\le [u_k]_{s,p}+[u]_{s,p}=[u_k]_{s,p}\to0.
\]
Together with the strong $L^p$ convergence, this implies that $u_k\to u$ strongly in $HW^{s,p}(\Sp)$. Because $\mathcal X$ is closed and each $u_k$ belongs to $\mathcal X$, we obtain $u\in\mathcal X$. By assumption, $\mathcal X$ contains no non-zero constant, hence $u\equiv0$. This contradicts
\[
\|u\|_{L^p(\Sp)}=\lim_{k\to\infty}\|u_k\|_{L^p(\Sp)}=1.
\]
The contradiction proves \eqref{eq:constraint-Lp}.

Applying Proposition \ref{prop:cont-embed} and then \eqref{eq:constraint-Lp}, we find
\[
\|u\|_{L^{\qexp}(\Sp)}
\le
C_S\bigl([u]_{s,p}+\|u\|_{L^p(\Sp)}\bigr)
\le
C_S(1+C_{\mathcal X})[u]_{s,p}.
\]
This is the first estimate in \eqref{eq:constraint-Lq}; the second follows by raising both sides to the power $p$.
\end{proof}

The preceding theorem immediately yields the full admissible sets on any coercive finite-codimensional class.

\begin{corollary}\label{cor:constraint-all-real}
Let $\mathcal X$ be as in Theorem \ref{thm:general-constraint}. Then
\[
\mathbb B_p(\mathcal X)=\mathbb R,
\qquad
\widehat{\mathbb B}_p(\mathcal X)=\mathbb R.
\]
\end{corollary}

\begin{proof}
Fix $B\in\mathbb R$ and write $B_-:=\max\{-B,0\}$. By Theorem \ref{thm:general-constraint},
\[
\|u\|_{L^{\qexp}(\Sp)}\le C_{\mathcal X}'[u]_{s,p}
\qquad\text{and}\qquad
\|u\|_{L^p(\Sp)}\le C_{\mathcal X}[u]_{s,p}
\qquad\text{for all }u\in\mathcal X.
\]
Hence
\[
\bigl(C_{\mathcal X}'+B_-C_{\mathcal X}\bigr)[u]_{s,p}+B\|u\|_{L^p(\Sp)}
\ge
C_{\mathcal X}'[u]_{s,p}
\ge
\|u\|_{L^{\qexp}(\Sp)},
\]
so $B\in \mathbb B_p(\mathcal X)$. The same argument with $p$-th powers yields
\[
\bigl((C_{\mathcal X}')^p+B_-C_{\mathcal X}^p\bigr)[u]_{s,p}^p+B\|u\|_{L^p(\Sp)}^p
\ge
(C_{\mathcal X}')^p [u]_{s,p}^p
\ge
\|u\|_{L^{\qexp}(\Sp)}^p,
\]
and therefore $B\in \widehat{\mathbb B}_p(\mathcal X)$.
\end{proof}

Two important special cases are worth recording explicitly.

\begin{corollary}\label{cor:zero-average}
Let
\[
\mathcal Z:=\left\{u\in HW^{s,p}(\Sp):\int_{\Sp}u\,\dd V_{\theta_0}=0\right\}.
\]
Then
\[
\mathbb B_p(\mathcal Z)=\mathbb R,
\qquad
\widehat{\mathbb B}_p(\mathcal Z)=\mathbb R.
\]
In particular, there exists $C_0>0$ such that
\[
\|u\|_{L^{\qexp}(\Sp)}\le C_0 [u]_{s,p},
\qquad
\|u\|_{L^{\qexp}(\Sp)}^p\le C_0^p [u]_{s,p}^p
\qquad\text{for all }u\in\mathcal Z.
\]
\end{corollary}

\begin{proof}
Consider the continuous linear functional
\[
L(u)=\int_{\Sp}u\,\dd V_{\theta_0}.
\]
Then $\mathcal Z=\ker L$ is a closed hyperplane of $HW^{s,p}(\Sp)$ and therefore has codimension one. Since $\mathcal Z$ contains no non-zero constant, Theorem \ref{thm:general-constraint} applies. Corollary \ref{cor:constraint-all-real} then yields the admissible-set statement.
\end{proof}

\begin{corollary}\label{cor:finite-orthogonality}
Let $\mathcal Y\subset C^\infty(\Sp)$ be a finite-dimensional subspace containing the constant function $1$, and define
\[
\mathcal Z_{\mathcal Y}
:=
\left\{
u\in HW^{s,p}(\Sp):
\int_{\Sp}u\psi\,\dd V_{\theta_0}=0
\ \text{for all }\psi\in\mathcal Y
\right\}.
\]
Then
\[
\mathbb B_p(\mathcal Z_{\mathcal Y})=\mathbb R,
\qquad
\widehat{\mathbb B}_p(\mathcal Z_{\mathcal Y})=\mathbb R.
\]
In particular, if
\[
\mathcal Y=\mathrm{span}\{1,\xi_1,\dots,\xi_{2n+2}\},
\]
then every function orthogonal to the constants and to the first-order spherical harmonics satisfies a pure seminorm inequality.
\end{corollary}

\begin{proof}
Choose a basis $\psi_1,\dots,\psi_m$ of $\mathcal Y$. For each $j$, the map
\[
L_j(u):=\int_{\Sp}u\psi_j\,\dd V_{\theta_0}
\]
is a continuous linear functional on $HW^{s,p}(\Sp)$, because $\psi_j$ is smooth and $HW^{s,p}(\Sp)\hookrightarrow L^p(\Sp)$. Hence
\[
\mathcal Z_{\mathcal Y}=\bigcap_{j=1}^m \ker L_j
\]
is a closed linear subspace of finite codimension. Since $1\in \mathcal Y$, the class $\mathcal Z_{\mathcal Y}$ contains no non-zero constant. Theorem \ref{thm:general-constraint} and Corollary \ref{cor:constraint-all-real} therefore give the conclusion.
\end{proof}

The preceding discussion may be summarized by a simple structural dichotomy.

\begin{theorem}\label{thm:linear-dichotomy}
Let $\mathcal X\subset HW^{s,p}(\Sp)$ be a finite-codimensional closed linear subspace. Then the following assertions are equivalent:
\begin{enumerate}[label=\textup{(\roman*)}]
\item $\mathcal X$ contains no non-zero constant function.
\item There exists $C>0$ such that
\[
\|u\|_{L^p(\Sp)}\le C [u]_{s,p}
\qquad\text{for all }u\in\mathcal X.
\]
\item There exists $C'>0$ such that
\[
\|u\|_{L^{\qexp}(\Sp)}\le C' [u]_{s,p}
\qquad\text{for all }u\in\mathcal X.
\]
\item $0\in \mathbb B_p(\mathcal X)\cap \widehat{\mathbb B}_p(\mathcal X)$.
\item $\mathbb B_p(\mathcal X)=\widehat{\mathbb B}_p(\mathcal X)=\mathbb R$.
\end{enumerate}
\end{theorem}

\begin{proof}
The implication \textup{(i)}$\Rightarrow$\textup{(v)} is exactly Corollary \ref{cor:constraint-all-real}, and \textup{(v)}$\Rightarrow$\textup{(iv)} is immediate. If \textup{(iv)} holds, then in particular there exists $C'>0$ such that
\[
\|u\|_{L^{\qexp}(\Sp)}\le C' [u]_{s,p}
\qquad\text{for all }u\in\mathcal X,
\]
which is \textup{(iii)}. Since $\Sp$ has finite measure, \textup{(iii)} implies \textup{(ii)} by H\"older's inequality:
\[
\|u\|_{L^p(\Sp)}
\le
\om^{1/p-1/\qexp}\|u\|_{L^{\qexp}(\Sp)}
\le
\om^{1/p-1/\qexp} C' [u]_{s,p}.
\]
Finally, if \textup{(ii)} held while $\mathcal X$ contained a non-zero constant $c$, then
\[
0<\|c\|_{L^p(\Sp)}\le C[c]_{s,p}=0,
\]
which is impossible. Thus \textup{(ii)} implies \textup{(i)}.
\end{proof}

\begin{remark}
The nonlinear first-moment class and the linear classes considered above illustrate the basic mechanism behind the lower-order term. Constraints that preserve constants leave the sharp lower-order coefficient unchanged; by contrast, finite-codimensional linear constraints that remove constants yield a coercive inequality and hence eliminate the need for any lower-order term.
\end{remark}

\section{Subcritical consequences and concluding remarks}\label{sec:consequences}

The critical inequalities obtained above admit a standard subcritical interpolation family. Although these estimates are not the main focus of the paper, they provide a useful bridge between the critical endpoint and the compact subcritical embeddings, and they also transfer exactly to the weighted Heisenberg formulation.

\begin{corollary}\label{cor:subcritical}
Let $r\in[p,\qexp)$. Then the following statements hold.
\begin{enumerate}[label=\textup{(\roman*)}]
\item For every $\varepsilon>0$ there exists $C_{\varepsilon,r}>0$ such that
\begin{equation}\label{eq:subcritical-1}
\|u\|_{L^r(\Sp)}
\le \varepsilon [u]_{s,p}+C_{\varepsilon,r}\|u\|_{L^p(\Sp)}
\qquad\text{for all }u\in HW^{s,p}(\Sp).
\end{equation}
\item For every $\varepsilon>0$ there exists $\widetilde C_{\varepsilon,r}>0$ such that
\begin{equation}\label{eq:subcritical-2}
\|u\|_{L^r(\Sp)}^p
\le \varepsilon [u]_{s,p}^p+\widetilde C_{\varepsilon,r}\|u\|_{L^p(\Sp)}^p
\qquad\text{for all }u\in HW^{s,p}(\Sp).
\end{equation}
\end{enumerate}
\end{corollary}

\begin{proof}
Fix $r\in[p,\qexp)$ and choose $\theta\in(0,1)$ such that
\[
\frac1r=\frac{\theta}{\qexp}+\frac{1-\theta}{p}.
\]
By interpolation,
\[
\|u\|_{L^r(\Sp)}
\le
\|u\|_{L^{\qexp}(\Sp)}^\theta \|u\|_{L^p(\Sp)}^{1-\theta}.
\]
Using Theorem \ref{thm:B-endpoint}, we obtain
\[
\|u\|_{L^{\qexp}(\Sp)}
\le
A_0 [u]_{s,p}+\om^{-s/Q}\|u\|_{L^p(\Sp)}.
\]
Since $0<\theta<1$, one has $(a+b)^\theta\le a^\theta+b^\theta$ for all $a,b\ge0$, and therefore
\[
\|u\|_{L^r(\Sp)}
\le
A_0^\theta [u]_{s,p}^\theta \|u\|_{L^p(\Sp)}^{1-\theta}
+
\om^{-\theta s/Q}\|u\|_{L^p(\Sp)}.
\]
Applying Young's inequality in the form
\[
x^\theta y^{1-\theta}
\le
\theta \delta x+(1-\theta)\delta^{-\theta/(1-\theta)}y,
\qquad x,y,\delta>0,
\]
we deduce that for every $\varepsilon>0$,
\[
A_0^\theta [u]_{s,p}^\theta \|u\|_{L^p(\Sp)}^{1-\theta}
\le
\varepsilon [u]_{s,p}+C_{\varepsilon,r}^{(1)}\|u\|_{L^p(\Sp)}.
\]
This proves \eqref{eq:subcritical-1}. To derive \eqref{eq:subcritical-2}, apply \eqref{eq:subcritical-1} with
\[
\varepsilon_1=\left(\frac{\varepsilon}{2^{p-1}}\right)^{1/p}
\]
and then use
\[
(a+b)^p\le 2^{p-1}(a^p+b^p).
\]
After renaming the constant, we obtain \eqref{eq:subcritical-2}.
\end{proof}

\begin{corollary}\label{cor:subcritical-Hn}
Let $r\in[p,\qexp)$. For every $\varepsilon>0$ there exist constants $C_{\varepsilon,r},\widetilde C_{\varepsilon,r}>0$ such that
\[
\|U\|_{L^r(J_c)}
\le
\varepsilon [U]_{K_c,p}+C_{\varepsilon,r}\|U\|_{L^p(J_c)},
\]
and
\[
\|U\|_{L^r(J_c)}^p
\le
\varepsilon [U]_{K_c,p}^p+\widetilde C_{\varepsilon,r}\|U\|_{L^p(J_c)}^p
\]
for all $U\in HW^{s,p}_{K_c}(\Hn)$.
\end{corollary}

\begin{proof}
Let
\[
u:=U\circ \Psi_c
\]
on $\Sp\setminus\{S\}$. By Lemma \ref{lem:cayley-isometry}, the identities
\[
\|u\|_{L^r(\Sp)}=\|U\|_{L^r(J_c)},
\qquad
\|u\|_{L^p(\Sp)}=\|U\|_{L^p(J_c)},
\qquad
[u]_{s,p}=[U]_{K_c,p}
\]
hold. Thus Corollary \ref{cor:subcritical} applied to $u$ is exactly equivalent to the stated inequalities for $U$.
\end{proof}

\begin{remark}
The analysis in this paper leaves several natural directions for further investigation. The first is the endpoint problem for nonlinear constrained classes that still contain constants, in particular the class $\mathcal M$ when $2<p<Q$. The second is the determination of the optimal leading coefficient $A$ in the critical inequalities, which would complement the lower-order classification obtained here. A third direction is to extend the compact argument to more general compact strictly pseudoconvex pseudohermitian manifolds and to related fractional operators defined spectrally from the sub-Laplacian.
\end{remark}

\section*{Acknowledgments}

%The authors would like to thank the anonymous referee for a careful reading of the manuscript and for helpful comments and suggestions.

\textbf{Funding.}
This work is supported by National Natural Science Foundation of China (12301145, 12261107, 12561020) and Yunnan Fundamental Research Projects (202301AU070144, 202401AU070123).

\textbf{Author contributions.}
The authors jointly developed the results and wrote the manuscript. All authors contributed equally to this work.

\textbf{Data availability.}
Data sharing is not applicable to this article, since no new data were created or analyzed in this study.

{\bf Conflict of Interest:} The authors declare that they have no conflict of interest.
\bibliographystyle{plain}
\bibliography{ref}

\end{document}